\documentclass[10pt, leqno, a4paper]{amsart}
\usepackage[latin1]{inputenc}
\usepackage{amsfonts}
\usepackage{amsmath}
\usepackage{amssymb}
\usepackage{amsthm}
\usepackage[T1]{fontenc}
\usepackage[dvips]{graphicx}
\usepackage{enumerate}
\usepackage{lmodern}
\usepackage{hyperref}
\usepackage[matrix, arrow]{xy}
\usepackage{listings}

\title{On the behaviour of strong semistability in geometric deformations}
\author{Holger Brenner}
\author{Axel St\"abler}

\DeclareMathOperator{\Proj}{Proj}

\DeclareMathOperator{\Syz}{Syz}
\DeclareMathOperator{\rk}{rk}

\DeclareMathOperator{\Spec}{Spec}

\newcommand{\primeboundlarge}{3433}

\newcommand{\primeexceptions}{$2$, $103$, $151$, $199$, $239$, $241$, $257$, $263$, $281$, $311$, $313$, $337$, $367$, $401$, $409$, $433$, $457$, $577$, $601$, $641$, $647$, $673$, $719$, $727$, $743$, $809$, $823$, $881$, $887$, $911$, $919$, $937$, $953$, $967$, $977$, $1033$, $1129$, $1153$, $1217$, $1249$, $1279$, $1303$, $1327$, $1399$, $1409$, $1423$, $1471$, $1481$, $1487$, $1511$, $1543$, $1583$, $1601$, $1607$, $1609$, $1663$, $1697$, $1759$, $1801$, $1823$, $1831$, $1871$, $1873$, $1913$, $1951$, $1993$, $1999$, $2063$, $2081$, $2087$, $2089$, $2113$, $2137$, $2153$, $2161$, $2207$, $2287$, $2297$, $2311$, $2399$, $2441$, $2447$, $2473$, $2503$, $2543$, $2591$, $2609$, $2617$, $2633$, $2713$, $2719$, $2801$, $2833$, $2879$, $2927$, $2969$, $3001$, $3023$, $3041$, $3049$, $3089$, $3167$, $3191$, $3217$, $3257$, $3313$, $3343$, $3359$}

\newcommand{\primeexceptionsnumber}{$105$} 
\begin{document}
\swapnumbers
\theoremstyle{definition}
\newtheorem{Le}{Lemma}[section]
\newtheorem{Def}[Le]{Definition}
\newtheorem*{DefB}{Definition}
\newtheorem{Bem}[Le]{Remark}
\newtheorem{Ko}[Le]{Corollary}
\newtheorem{Theo}[Le]{Theorem}
\newtheorem*{TheoB}{Theorem}
\newtheorem{Bsp}[Le]{Example}
\newtheorem{Be}[Le]{Observation}
\newtheorem{Prop}[Le]{Proposition}
\newtheorem{Sit}[Le]{Situation}
\newtheorem{Que}[Le]{Question}
\newtheorem*{Con}{Conjecture}
\newtheorem{Dis}[Le]{Discussion}
\newtheorem{Prob}[Le]{Problem}
\newtheorem{Konv}[Le]{Convention}
\def\cocoa{{\hbox{\rm C\kern-.13em o\kern-.07em C\kern-.13em o\kern-.15em
A}}}
\address{Holger Brenner\\
Universit\"at Osnabr\"uck, Fachbereich 6: Mathematik/Informatik,
Albrechtstr. 28a,
49069 Osnabr\"uck, Germany}
\email{hbrenner@uni-osnabrueck.de}

\address{Axel St\"abler
Universit\"at Osnabr\"uck, Fachbereich 6: Mathematik/Informatik,
Albrechtstr. 28a,
49069 Osnabr\"uck, Germany}
\curraddr{Johannes Gutenberg-Universit\"at Mainz\\ Fachbereich 08\\
Staudingerweg 9\\
55099 Mainz\\
Germany}
\email{axel.staebler@uni-osnabrueck.de}

\begin{abstract}
Let $Y \to B$ be a relative smooth projective curve over an affine integral base scheme $B$ of positive characteristic. We provide for all prime characteristics example classes of vector bundles $\mathcal{S}$ over $Y$ such that $\mathcal{S}$ is generically strongly semistable and semistable but not strongly semistable for some special fibre.
\end{abstract}
\maketitle

\section*{Introduction}

Let $B$ be an affine base scheme over a field $k$ of positive characteristic $p$ and consider a relative smooth projective curve $Y \rightarrow B$. Let $\mathcal S$ be a vector bundle over $Y$ so that every point $t \in B$ induces a vector bundle ${\mathcal S}_t$ on the fibre $Y_t$. We are interested in the question how the property of ${\mathcal S}_t$ being strongly semistable varies with $t$. Recall that a vector bundle on a smooth projective curve over a field of positive characteristic is called strongly semistable if all its Frobenius pull-backs are semistable.

The set of base points $t \in B$ such that the bundle ${\mathcal S}_t$ is semistable is open (possibly empty) -- see \cite[Theorem 2.8]{maruyamasemistableopen}. As we are in equal characteristic $p$, the $e$th Frobenius pull-back $F^{e^\ast} {\mathcal S}$ is again a bundle on the given curve and the semistability property of its induced bundles on the fibres also defines an open subset. Therefore the set of points parametrising strongly semistable bundles is a countable intersection of open subsets. In characteristic $2$ and $3$, P. Monsky has given examples (in the language of Hilbert-Kunz theory) over the affine line (minus some points) such that the bundle is strongly semistable over the generic point but for no closed point (see \cite{monskypoints4quartics}, \cite{monskyzdp4}).

In this paper we want to provide in all characteristics new example classes of bundles over $Y \rightarrow B$, where $B$ is an affine smooth curve of finite type over $k$, such that the generic bundle ${\mathcal S}_\eta$ over $Y_\eta$ is strongly semistable for the generic point $\eta \in B$ and semistable but not strongly semistable for some closed point $t \in B$. By localising at $t$ we get then a bundle over a relative curve over a discrete valuation domain such that the generic bundle is strongly semistable and such that the special bundle is semistable but not strongly semistable.

We describe here two constructions which lead to such examples. In both constructions we use syzygy bundles ${\mathcal S} = \Syz$, i.\,e.\ bundles given by short exact sequences of the form
\[0 \longrightarrow \Syz(f_1, \ldots, f_n) \longrightarrow \bigoplus_{i=1}^n {\mathcal O}_Y(-d_i) \stackrel{f_1, \ldots , f_n}{\longrightarrow} {\mathcal O}_Y \longrightarrow 0\, .\]
Here the $f_i$ are homogeneous elements of degree $d_i$ in a graded ring $R$ with $Y= \Proj R$ which are primary to the maximal graded ideal. Syzygy bundles exhibit a rich behaviour yet they are accessible for computations. In the end, we will work with syzygy bundles of rank two (given by three generators) on plane curves. One should also remark that on a smooth curve over an algebraically closed field every vector bundle is, up to twist, a syzygy bundle (see e.\,g.\ \cite[Proposition 3.8]{brennerstaeblerdaggersolid}).

In the first construction (Section \ref{fixedcurve}) the curve family will be trivial, i.\,e.\ a product $Y_0 \times B \rightarrow B$ where $Y_0$ is a smooth projective curve over $k$, and the syzygy bundle will be given by homogeneous elements varying with the basis. With this construction we provide explicit examples in all characteristics for Fermat curves of some degree $\delta \geq 5$.

In the second construction (Sections \ref{varycurvegeneric} and \ref{varycurvespecial} after some prepatory work in Section \ref{criteria}) the syzygy bundle is defined on the projective plane ${\mathbb P}^2$ and the family arises by restricting to a certain family of plane curves of degree $4$. Due to computational constraints, we can provide with this method only examples for odd characteristics $p \leq \primeboundlarge$ with \primeexceptionsnumber\ exceptions, the smallest one being $103$.

We thank V.\ Mehta for useful discussions and D.\ Brinkmann for proof-reading an earlier version of this article.

\section{Deforming syzygy bundles on a fixed curve}
\label{fixedcurve}

Suppose that we have a fixed smooth projective curve $Y$ over an algebraically closed field $k$ of positive characteristic. The moduli space $\mathcal M$ of semistable vector bundles of given rank and degree \cite{lepotier} may also contain points representing stable but not strongly semistable bundles. In fact, by \cite[Theorem 1]{langepaulyfrobenius} such bundles exist for all curves of genus $\geq 2$. If we connect such a point with a point representing a strongly semistable bundle by an integral affine curve $B$, then the universal bundle on $\mathcal M$ (if it exists) induces a bundle $\mathcal S$ on $Y \times B$ such that for some closed point $b_0 \in B$ the bundle ${\mathcal S}_0$ is strongly semistable and such that for some closed point $b_1 \in B$ the bundle ${\mathcal S}_1$ is semistable but not strongly semistable. In this case the bundle ${\mathcal S}_\eta$ on the generic curve $Y_\eta$ over $\kappa(\eta)$ is also strongly semistable.

We are particularly interested in bundles of degree $0$. So there is no universal bundle. One can then still use moduli constructions to obtain such examples. Specifically, one can use the existence of quasi universal families (cf.\ \cite[Proposition 4.6.2]{huybrechtslehn}) or one can use the existence of quasi sections (cf.\ \cite[I.2.25 and I.3.26]{milne}) and apply the construction outlined above to the Quot scheme from which the moduli space is constructed (cf.\ also \cite[Chapter 4]{huybrechtslehn}).

However, we would like to get concrete examples with such a behaviour. Therefore, we work instead directly with syzygy bundles and deform their defining sections. Fix an ample line bundle ${\mathcal O}_Y(1)$ on $Y$ and a degree type $(d_1 , \ldots ,d_n)$. For a tuple of global sections $f_i \in  \Gamma(Y, {\mathcal O}_Y(d_i))$ we get a syzygy bundle ${\mathcal S}_0 = \Syz (f_1, \ldots , f_n)$ provided that the sections determine a surjective map $\oplus_i {\mathcal O}_Y(-d_i) \rightarrow {\mathcal O}_Y $. If $g_1, \ldots ,g_n$ is another such tuple (with the same degree type and syzygy bunde ${\mathcal S}_1 = \Syz (g_1, \ldots , g_n)$), then we can consider for $t_1, \ldots, t_n \in k$ the family 
\[ \Syz (t_1 g_1+ (1-t_1)f_1, \ldots, t_n g_n+ (1-t_n) f_n)\]
of syzygy bundles on $Y \times {\mathbb A}^n \rightarrow {\mathbb A}^n$ (or over an open subset of ${\mathbb A}^n$). For $t=(t_1, \ldots, t_n)=(0, \ldots, 0)$ the fibre is ${\mathcal S}_0$ and for $t=(1, \ldots ,1)$ the fibre is ${\mathcal S}_1$. Note that one should look at this family over an open subset of ${\mathbb A}^n$ to make sure that the corresponding morphisms are surjective for every parameter (this is not automatically fulfilled, as the example $\Syz(x,y)$ and $\Syz (y,x)$ on ${\mathbb P}^1_{k}$, where $k$ is a field of characteristic $\neq 2$, shows). Also, if ${\mathcal S}_0$ and ${\mathcal S}_1$ are semistable, one might shrink the open subset further to ensure that all bundles in the family are semistable. 

The following simple instance of this construction yields already a large class of examples of families with the properties described in the introduction. We look at examples of semistable rank two syzygy bundles on Fermat curves and we deform one parameter of the defining sections to get the trivial bundle. 

\begin{Theo}
\label{main1}
Let $\delta \geq 5$ be a natural number and let $p$ be a prime number such that there exists $e$ so that $2 r^{e}  < \delta  < 3r^{e}  $, where $r=p \mod \delta $, and where $r^e$ is considered modulo $\delta$ as the least non-negative representative. Let $k$ be a field of characteristic $p$ and
let \[\mathcal{S} = \Syz(x^{2}, y^{2}, tz^{2} + (1-t)xy)(3) \] 
be the syzygy bundle on the smooth relative curve
\[ Y = \Proj k[t][x,y,z]/(x^\delta  + y^\delta  - z^\delta ) \longrightarrow \Spec k[t] \, . \]
Then for $t=1$ the special bundle ${\mathcal S}_1$ is semistable but not strongly semistable
and for $t=0$ the special bundle ${\mathcal S}_0$ is trivial, hence strongly semistable. For the generic point $\eta \in \Spec k[t]$ the bundle ${\mathcal S}_\eta$ on $Y_\eta$ is strongly semistable.
\end{Theo}
\begin{proof}
First of all, note that $Y \to \Spec k[t]$ is flat by \cite[Proposition III.9.7]{hartshornealgebraic}. Furthermore, the sheaf of relative differentials is locally free so that the morphism is indeed smooth.
The specialisation $t = 1$ yields the syzygy bundle $\Syz(x^2, y^2, z^2)(3)$ which is semistable but not strongly semistable under the given numerical condition by virtue of \cite[Corollary 2]{brennermiyaoka}.

The specialisation $t = 0$ yields $\Syz(x^{2}, y^{2}, xy)(3)$ which is already defined on $\mathbb{P}^1_k$ and splits as $\mathcal{O}_{Y_0}^2$ (the trivialising sections are the syzygies $(0, x, -y), (y,0,-x)$). Therefore $\mathcal{S}_0$ is strongly semistable and hence $\mathcal{S}_\eta$ is also strongly semistable by the openness of semistability.
\end{proof}

The bundles in this family have rank $2$ and degree $0$. We do not know whether the bundles  are semistable for all $t$, but they are for an open subset. The locus of points parametrizing  strongly semistable bundles need not be open. Note that the generic bundle in this family is not trivial. Indeed, the specialisation at $t=1$ does not have a non-trivial global section and by semicontinuity (\cite[Theorem III.12.8]{hartshornealgebraic}) this also holds generically.

If $\delta $ is a prime number $\geq 5$ then there do exist prime numbers $p$ fulfilling the given numerical condition. On the other hand, for every prime number $p$ there exist natural numbers $\delta$ with $2p < \delta  < 3p$ and $\delta  \geq 5$. Hence, Theorem \ref{main1} provides examples for all characteristics.

Localising we immediately obtain

\begin{Ko}
Let $\delta  \geq 5$ be a natural number and let $p$ be a prime number such that there exists $e$ so that $2 r^{e}  < \delta  < 3r^{e}  $, where $r=p \mod \delta $. Let $k$ be a field of characteristic $p$ and let $V= k[t]_{(t - 1)}$. Then the syzygy bundle 
\[\mathcal{S} = \Syz(x^{2}, y^{2}, tz^{2} + (1-t)xy)(3) \] 
on the smooth relative curve
\[ Y = \Proj V [x,y,z]/(x^\delta  + y^\delta  - z^\delta ) \longrightarrow \Spec V \]
has  strongly semistable generic fibre and semistable but not strongly semistable
special fibre.
\end{Ko}

\begin{Bsp}
For $\delta =5$ exactly the prime characteristics $p$ with $p=2,3 \mod 5$ fulfill the numerical condition of Theorem \ref{main1}. For $p=1,4 \mod 5$ we do not know whether there are semistable but not strongly semistable fibres in the family.
\end{Bsp}

\section{Sufficient criteria for semistability}
\label{criteria}

In this section we provide several sufficient criteria for deciding whether a given vector bundle on a curve is (strongly) semistable. Similar considerations also occur in \cite[Chapter 3]{kaiddiss}. We say that a global section $s \in \Gamma (Y,\Syz(m))$ is a syzygy of total degree $m$.

\begin{Le}
\label{LGlobalSectionsStrongSst}
Let $Y$ be a smooth projective curve of genus $g > 0$ over a field $k$ of positive characteristic $p > 0$ with an ample line bundle $\mathcal{O}_Y(1)$. Let $\mathcal{S}$ be a vector bundle of slope $\mu$. Then $\mathcal{S}$ is strongly semistable if and only if $({F^e}^\ast \mathcal{S})(m)$ has no global section for $m < -\frac{\mu}{\deg \mathcal{O}_Y(1)} q$ for all $q = p^e$. 

If moreover, $\mathcal{O}_Y(1)$ has a non-trivial global section then it is enough that for all $q=p^e$ the bundle $({F^e}^\ast \mathcal{S})(m)$ has no non-trivial global sections, where $m = -([\frac{\mu}{\deg \mathcal{O}_Y(1)}q] +1)$.
\end{Le}
\begin{proof}
If $F^{e^\ast} \mathcal{S}(m)$ has a non-trivial global section, with $m < -\frac{\mu}{\deg \mathcal{O}_Y(1)} q$, then we obtain an exact sequence $0 \to \mathcal{O}_Y(-m) \to F^{e^\ast} \mathcal{S}$ contradicting semistability.

Conversely, let $0 \to \mathcal{E} \to F^{e^\ast} \mathcal{S}$ be exact, where $\mathcal{E}$ is a bundle with $\mu(\mathcal{E}) > q \mu$. Let $q' = p^{e'}$ be such that $q'(\mu(\mathcal{E})-q \mu) \geq g + \deg \mathcal{O}_Y(1) $, and let $l$ be such that
\[- \deg \mathcal{O}_Y(1) \leq \mu ( F^{(e+e')^\ast} \mathcal{S} )(l)) = p^{e+e'} \mu + l \deg \mathcal O_Y(1)  < 0,\] i.\,e.\, $-1 - \frac{p^{e+e'}\mu}{\deg \mathcal{O}_Y(1)} \leq l < - \frac{p^{e+e'}\mu}{\deg \mathcal{O}_Y(1)}$. We have the injection
$(F^{e'^\ast} \mathcal{E}) (l) \to ( F^{(e+e')^\ast} \mathcal{S} )(l)$ and
\[ \mu  ((F^{e'^\ast} \mathcal{E}) (l) ) = q' \mu({\mathcal E}) + l \deg {\mathcal O}_Y(1) \geq q' \mu ({\mathcal E}) -\deg {\mathcal O}_Y(1)  -p^{e+e'} \mu \geq g. \]
By Riemann-Roch for vector bundles (see e.\,g.\ \cite[Theorem 7.D.3]{patilstorchalggeo}) applied to $\mathcal{T}=(F^{e'^\ast}\mathcal{E})(l)$ we therefore obtain $\frac{\chi(\mathcal{T})}{\rk \mathcal{T}} = \mu(\mathcal{T}) + \chi(\mathcal{O}_Y) \geq 1$. It follows that $\mathcal{T}$ and therefore also $(F^{(e+e')^\ast} \mathcal{S})(l)$ has a non-trivial global section.

If $\mathcal{O}_Y(1)$ and $\mathcal{S}(m)$ both have a non-trivial global section then a fortiori $\mathcal{S}(m+1)$. Hence, in this case it is enough to check that $F^{e^\ast} \mathcal{S}$ has no global sections in degree $m$, where $m$ is the maximal $m \in \mathbb{Z}$ such that $\mu((F^{e^\ast} \mathcal{S})(m)) < 0$. 
\end{proof}

We will call the degree occuring in the last statement of the previous lemma the critical degree. A similar proof yields the following sufficient condition for semistability.

\begin{Le}
\label{KGlobalSectionsSst}
Let $Y$ be a smooth projective curve of genus $g >0$ over a field $k$ of positive characteristic $p > 0$ with an ample line bundle $\mathcal{O}_Y(1)$. Let $\mathcal{S}$ be a vector bundle of slope $\mu$. Suppose that $({F^{e}}^\ast \mathcal{S})(m)$ has no global sections for $m < -\frac{\mu}{\deg \mathcal{O}_Y(1)}q$ for some $q = p^e \geq g + \deg \mathcal{O}_Y(1)$. 
Then $\mathcal{S}$ is semistable. 

If moreover, $\mathcal{O}_Y(1)$ has a non-trivial global section then $\mathcal{S}$ is semistable if $(F^{e^\ast} \mathcal{S})(m)$ has no global sections, where $m = -([\frac{\mu}{\deg \mathcal{O}_Y(1)}q] +1)$ for some $q = p^e \geq g + \deg \mathcal{O}_Y(1)$.
\end{Le}

For a plane curve of degree $\delta$ the bound in the previous lemma becomes $q \geq \frac{1}{2} \delta(\delta -1) +1$.

\begin{Le}
\label{HNFMethod}
Let $Y$ be a smooth projective curve over a field $k$ of positive characteristic $p > 0$ with a fixed ample line bundle $\mathcal{O}_Y(1)$. Let $\mathcal{S}$ be a vector bundle of rank $2$ on $Y$. If $(F^\ast \mathcal{S})(m)$ has a global section $s$ without zeros, where $m < - \frac{p \deg \mathcal{S}}{2 \deg \mathcal{O}_Y(1)}$ and $p \nmid m \deg \mathcal{O}_Y(1)$ then $\mathcal{S}$ is semistable but not strongly semistable. 
\end{Le}
\begin{proof}
First of all, since $\mathcal{S}$ has rank two any Harder-Narasimhan filtration is automatically strong. Moreover, $0 \to \mathcal{O}_Y(-m) \xrightarrow{\cdot s} F^\ast \mathcal{S}$ is a strong Harder Narasimhan filtration of $F^\ast \mathcal{S}$. By the second assumption $\mathcal{O}_Y(-m)$ cannot be a Frobenius pullback of another line bundle. Hence, $\mathcal{S}$ is semistable.
\end{proof}

Note that this method is in principle also applicable for smooth surfaces since we may pass to the reflexive hull of a destabilising subsheaf which is locally free (\cite[Corollary 1.4]{hartshornestablereflexive}).

\section{Restricting a syzygy bundle to varying curves - generic case}
\label{varycurvegeneric}

In this section we want to ``fix the syzygy bundle and deform the curve''. What we mean by this is that we fix a syzygy bundle $\mathcal{S}$ on ${\mathbb P}^2_k$ and study the restriction of $\mathcal{S}$ to a family of plane curves $Y_t \subset {\mathbb P}^2_k$ of a certain degree $\delta$. So $\mathcal{S} =\Syz(f_1, \ldots , f_n) $ with homogeneous polynomials $f_i \in k[x,y,z]$ and $Y_t =V_+(G_t)$ where $G_t \in k[t,x,y,z]$ is homogeneous with respect to $x,y,z$ and where $t$ is a new variable of degree zero. If $\mathcal{S}$ is semistable on ${\mathbb P}^2$, then it is natural to expect that for (very) generic $t$ the restriction to $Y_t$ is (strongly) semistable, but for specific $t$ anything may happen. There are many strong results in this direction (\cite{mehtaramanathanrestriction}, \cite{flennerrestriction}, \cite{bogomolovstability}, \cite{brennerstronglysemistable}, \cite{langersemistable}). In fact, \cite[Theorem 4.1]{langerrestrictionsemistable} shows that the restriction of $\mathcal{S}$ to a generic hypersurface in $\mathbb{P}^2_k$ of sufficiently high degree is strongly semistable.

We shall study the semistability properties of ${\mathcal S}_t =\Syz(f_1, \ldots , f_n)|_{Y_t}$ by looking at the existence of non-trivial global sections of the Frobenius pull-backs
\[ F^{e^\ast}\mathcal{S}_t \cong \Syz(f_1^q, \ldots , f_n^q)|_{Y_t} \, . \]
($q=p^e$) in certain critical degree twists of the bundle. The existence of such non-trivial global sections is equivalent to the property that certain systems of linear equations have a non-zero solution, which in turn depends on the (non)vanishing of certain determinants. These determinants will be non-zero polynomials in $t$, but for certain values of $t$ they will have a zero producing a (strong) semistability behaviour different from the generic behaviour.

We will mainly work with the syzygy bundles $\Syz(x^a,y^a,z^a)$ and their restrictions to a family of plane curves, where we allow transcendental coefficients. For $a=1$ this is directly related to the Hilbert-Kunz multiplicity of the corresponding coordinate ring.

\begin{Theo}
\label{PGenericStronglySst}
Let $k$ denote a field of positive characteristic $p > 0$ and consider the smooth generic plane projective curve $Y$ of even degree $\delta$ (prime to $p$) given by the homogeneous coordinate ring
\[ k(t)[x,y,z]/(x^\delta+ y^\delta +t x^{\frac{\delta}{2}}y^{\frac{\delta}{2}} - z^\delta) \,  .\]
Then the syzygy bundle $\Syz(x^a, y^a, z^a)$ ($a \in {\mathbb N}_+$ such that $\delta \nmid a$) is strongly semistable on $Y$.
\end{Theo}
\begin{proof}
By Lemma \ref{LGlobalSectionsStrongSst} it is enough to show that for every $q = p^e$ the syzygy bundle $\Syz(x^{aq}, y^{aq}, z^{aq})$ has no global sections of total degree $< \frac{3}{2} aq$.

Write $aq = \delta l + r$ with $0 < r < \delta$ (note that we have strict inequality due to the constraints we imposed on $a$ and $\delta$). By \cite[Lemma 1]{brennermiyaoka} global sections of minimal degree of $F^{e^\ast} \mathcal{S} = \Syz(x^{aq}, y^{aq},z^{aq})$ stem from global sections of \[\mathcal{F}_l = \Syz(x^{aq}, y^{aq}, (x^\delta + y^\delta + tx^{\frac{\delta}{2}}y^{\frac{\delta}{2}})^l)\] (add $r$ to the total degree) or of \[\mathcal{F}_{l+1} = \Syz(x^{aq}, y^{aq}, (x^\delta + y^\delta + tx^{\frac{\delta}{2}}y^{\frac{\delta}{2}})^{l+1})\] (of the same total degree) or a sum of such syzygies\footnote{A section $(s_1, s_2, s_3)$ of $\mathcal{F}_l$ yields the syzygy $(z^r s_1, z^r s_2, s_3)$ of $\Syz(x^{aq}, y^{aq}, z^{aq})$ and a section $(t_1, t_2, t_3)$ of $\mathcal{F}_{l+1}$ yields the syzygy $(t_1, t_2, z^{\delta -r} t_3)$ of $\Syz(x^{aq}, y^{aq}, z^{aq})$.}.
Set $f = x^\delta + y^\delta + tx^{\frac{\delta}{2}}y^{\frac{\delta}{2}}$. To determine the global sections of $\mathcal{F}_{l}$ and $\mathcal{F}_{l+1}$ we can work on the projective line ${\mathbb P}^1_{k(t)}$ or the ring $k(t)[x,y]$ respectively. The critical degree of $\mathcal{S}$ is $ \lceil \frac{3aq}{2} \rceil -1$ while the critical degree of $\mathcal{F}_l$ is $\lceil \frac{3aq -r }{2} \rceil - 1$ and that of $\mathcal{F}_{l+1}$ is $\lceil \frac{3aq +\delta - r }{2} \rceil - 1$. We will show that $\mathcal{F}_l$ and $\mathcal{F}_{l+1}$ have no non-trivial global sections of critical degree. Then also $F^{e^\ast}\mathcal{S}$ will not have non-trivial global sections of critical degree.

First we consider the case $\mathcal{F}_l$. Assume that there is a non-trivial syzygy $\alpha x^{aq} + \beta y^{aq} + \gamma f^l = 0$ of critical degree $m$. In particular, we must have $\gamma f^l \equiv 0 \mod (x^{aq}, y^{aq})$ (conversely, if this condition holds for some non-zero $\gamma$ then we obtain a non-trivial syzygy). That is, we must have that each coefficient $\omega_v$ (in $k[t]$) in
\[ \gamma f^l = \sum_v \omega_v x^v y^{m-v} \]
with $m - v < aq$ and $v < aq$ is zero. This is the case if and only if $w_v = 0$ for $m - aq + 1= \frac{\delta l}{2} \leq v \leq  aq -1 = m - \frac{\delta l}{2}$. Let us write \[ f^l = \sum_{i=0}^{\delta l} D_i x^i y^{\delta l - i} \text{ and } \gamma = \sum_{u=0}^{m - \delta l} c_u x^u y^{m - \delta l -u}. \] Then $\omega_v =D_v c_0 + \ldots + D_{v-m+ \delta l }  c_{m - \delta l } $. In particular, we set $D_i = 0$ if $i < 0$ or if $i > \deg f^l$. The condition that the $w_v$ be zero translates to a system of linear equations in the $c_u$ which we may write in the form of a square matrix of length $m-\delta l +1 =aq - \frac{\delta l}{2}$ as
\[\begin{pmatrix} w_{ \frac{\delta l}{2} } \\ w_{ \frac{\delta l}{2} +1 } \\ \vdots \\ w_{m - \frac{\delta l}{2} -1} \\ w_{m -  \frac{\delta l}{2}  } \end{pmatrix} \! = \!
\begin{pmatrix} D_{\frac{\delta l}{2}} & D_{\frac{\delta l}{2} -1 } & D_{\frac{\delta l}{2} -2} & \ldots & D_{\frac{3}{2} \delta l - m}\\ 
D_{\frac{\delta l}{2} + 1} &  D_{\frac{\delta l}{2}} &  D_{\frac{\delta l}{2} -1 } & \ldots & D_{\frac{3}{2} \delta l - m + 1}\\
\vdots & \vdots & \vdots & \ddots & \vdots\\
D_{m - \frac{\delta l}{2} - 1} & D_{m - \frac{\delta l}{2} - 2} & D_{m - \frac{\delta l}{2} - 3} & \ldots & D_{\frac{\delta l}{2} -1}\\
D_{m - \frac{\delta l}{2}} & D_{m - \frac{\delta l}{2} - 1} & D_{m - \frac{\delta l}{2} - 2} & \ldots & D_{\frac{\delta l}{2}} 
\end{pmatrix} \begin{pmatrix} c_0 \\ c_1 \\ \vdots \\ c_{m - \delta l - 1} \\ c_{m - \delta l} \end{pmatrix}. \]
The $i$th column of the matrix is obtained by taking the coefficients of $f^l$ whose degrees in $x$ range from $\frac{\delta l}{2} - i +1$ to $aq - i$.

This system has a non-trivial solution if and only if the determinant, call it $D^{(q)}$, of this square matrix is zero. Note that $D^{(q)}$ is a polynomial in $k[t]$. 

The coefficient $D_{\frac{\delta l}{2}}$ is of the form $t^l +$ terms of lower degree and the other entries $D_j$ for $j \neq \frac{\delta l}{2}$ are polynomials of degree $\leq l -1$ in $t$. Since the $D_{\frac{\delta l}{2}}$ occur exactly in the diagonal of the matrix we have that the determinant is a polynomial of degree $l(aq - \frac{\delta l}{2})$ in $t$. In particular, it is not the zero polynomial and hence invertible in $k(t)$. Therefore there is no non-trivial solution to the above system.

Next, we turn our attention to the case of $\mathcal{F}_{l+1}$. Again we denote the critical degree by $m$. By a similar argument to the previous case we obtain that there is a non-trivial syzygy $(\alpha, \beta, \gamma)$ of total degree $m$ if and only if the $w_v$ in $\gamma f^{l+1} = \sum_{v = 0} w_v x^v y^{m-v}$ are zero for $m - aq + 1 = \frac{\delta (l+1)}{2}\leq v \leq aq -1$. Again we write \[f^{l+1} = \sum_{i=0}^{\delta (l+1)} E_i x^i y^{\delta l - i} \text{ and } \gamma = \sum_{u=0}^{m - \delta(l+1)} c_u x^u y^{m - (l+1) \delta -u}. \] 
We set $E_i = 0$ for $i < 0$ or $i > \deg f^{l+1}$.
This translates to a system of linear equations in the $c_i$ which we write with the following square matrix of length $m-\delta ( l +1) +1 = \frac{aq - \delta + r}{2} = aq - \frac{\delta(l+1)}{2}$ as

\[\! \begin{pmatrix} E_{\frac{\delta (l+1)}{2}} & E_{\frac{\delta (l+1)}{2} - 1} & E_{\frac{\delta (l+1)}{2}- 2} & \ldots & E_{\frac{3}{2}\delta (l+1) - m}\\
E_{\frac{\delta (l+1)}{2} + 1} & E_{\frac{\delta (l+1)}{2}} & E_{\frac{\delta (l+1)}{2} - 1} & \ldots & E_{\frac{3}{2} \delta (l+1) - m + 1}\\
\vdots & \vdots & \vdots & \ddots & \vdots\\
E_{m - \frac{\delta(l+1)}{2} - 1} & E_{m - \frac{\delta(l+1)}{2} - 2} &E_{m - \frac{\delta(l+1)}{2} - 3} & \ldots & E_{\frac{\delta (l+1)}{2} - 1}\\
E_{m - \frac{\delta(l+1)}{2}} & E_{m - \frac{\delta(l+1)}{2} - 1}& E_{m - \frac{\delta(l+1)}{2} - 2} & \ldots & E_{\frac{\delta (l+1)}{2}}
\end{pmatrix} \! \begin{pmatrix} c_0 \\ c_1 \\ \vdots \\ c_{m - \delta (l+1) - 1} \\ c_{m - \delta (l+1)} \end{pmatrix}\!.\]

The coefficient $E_i$ is a polynomial in $t$ of degree $l+1$ if $i = \frac{\delta(l+1)}{2}$ and of degree $< l + 1$ otherwise. Hence, we obtain that the determinant, call it $E^{(q)}$, is a polynomial of degree $(l+1) (aq - \frac{\delta(l+1)}{2})$ in $t$. In particular, the matrix is invertible over $k(t)$. 
\end{proof}

\begin{Bem}
In the class of examples considered in Theorem \ref{PGenericStronglySst} we may assume that $\gcd(\frac{\delta}{2}, a) = 1$. For if $m r = \frac{\delta}{2}$
and $m b = a$ then we can obtain $\Syz(x^a, y^a, z^a)$ as the pull-back of $\Syz(u^b,
v^b, w^b)$ on $X = \Proj k[u,v,w]/(u^{2r} + v^{2r} + t u^r v^r - w^{2r})$ along the
mapping $u \mapsto x^m, v \mapsto y^m, w \mapsto y^m$. In particular, if $r = 1$
then $X \cong \mathbb{P}^1_k$ and semistability coincides with strong semistability.
\end{Bem}

\begin{Ko}
\label{PGenericStronglySstgeneric}
Let $k$ denote a field of positive characteristic $p > 0$ and consider the smooth generic plane projective curve $Y$ of even degree $\delta$ (prime to $p$) given by the homogeneous coordinate ring
\[ k(C_\nu)[x,y,z]/(\sum_{\vert\nu\vert = \delta} C_\nu x^{\nu_1} y^{\nu_2} z^{\nu_3}).\]
Then the syzygy bundle $\Syz(x^a, y^a, z^a)$ ($a \in {\mathbb N}_+$ such that $\delta \nmid a$) is strongly semistable on $Y$.
\end{Ko}
\begin{proof}
By Lemma \ref{LGlobalSectionsStrongSst} it is enough to show that for every $q = p^e$ the syzygy bundle $\Syz(x^{aq}, y^{aq}, z^{aq})$ has no global sections of total degree $< \frac{3}{2} aq$. Since the non-existence of non-trivial sections is an open property by semicontinuity (cf.\ \cite[Theorem III.12.8]{hartshornealgebraic} we may restrict to the curve given by \[z^\delta = x^\delta + y^\delta + t x^{\frac{\delta}{2}}y^{\frac{\delta}{2}}  \] over $k(t)$, where $t$ is an indeterminate. Then the result follows from Theorem \ref{PGenericStronglySst}.
\end{proof}

\begin{Bem}
Corollary \ref{PGenericStronglySstgeneric} can also be deduced (for $a=1$ and $\delta \geq 2$) from a result of Buchweitz and Chen (see \cite[Corollary 1]{buchweitzchenhilbertkunz}). They show that the Hilbert-Kunz multiplicity of the homogeneous coordinate ring of a generic plane curve of degree $\delta \geq 2$ is the minimal possible one. This is equivalent to the strong semistability of the syzygy bundle.

In fact, one can use this to show that $\mathcal{S} = \Syz(x^a,y^a, z^a)$ is strongly semistable on general plane curves of degree $l a$, $l \geq 2$. Consider the morphism $\varphi: \mathbb{P}_k^2 \to \mathbb{P}_k^2$ which maps $x,y,z$ to $x^a, y^a, z^a$. Then $\varphi^\ast \Syz(x,y,z) = \mathcal{S}$ and $\varphi^\ast \mathcal{O}_{\mathbb{P}_k^2}(1) = \mathcal{O}_{\mathbb{P}_k^2}(a)$. So if $\Syz(x,y,z)$ is strongly semistable restricted to some member of the linear system $\mathcal{O}_{\mathbb{P}_k^2}(l)$ then $\mathcal{S}$ will be strongly semistable restricted to the corresponding member of $\mathcal{O}_{\mathbb{P}_k^2}(la)$. 
\end{Bem}

\section{Restricting a syzygy bundle to varying curves - special case}
\label{varycurvespecial} 

We now study semistability properties of the restriction of $\Syz(x^a, y^a, z^a)$ to the curve given by $ z^\delta  = x^\delta + y^\delta + t_0 x^{\frac{\delta}{2}}y^{\frac{\delta}{2}} $ for special parameters $t_0 \in \overline{k}$. Also note that in order to obtain a smooth family we have to restrict ourselves to $t_0 \neq -\frac{\delta}{2}, \frac{\delta}{2}$.

The following lemma collects several criteria in terms of the determinants $D^{(q)}$ and $E^{(q)}$ that occured in the proof of Theorem \ref{PGenericStronglySst}.

\begin{Le}
\label{LPseudoAlgorithms}
Let $k$ be a field of characteristic $p >0 $ and denote by $\mathcal{S}$ the syzygy bundle $\Syz(x^a, y^a, z^a)$ on \[Y = \Proj k[t][x,y,z]/(x^\delta + y^\delta + t x^{\frac{\delta}{2}}y^{\frac{\delta}{2}} - z^\delta),\] where $\delta \nmid a$, $\delta$ even and $p \nmid \delta$. 
Write $aq = \delta l  + r$ with $0 < r < \delta $, where $q = p^e$, and denote by $D^{(q)}$ the determinant associated to \[\mathcal{F}_l = \Syz(x^{aq}, y^{aq}, (x^\delta + y^\delta + t x^{\frac{\delta}{2}} y^{\frac{\delta}{2}})^l)\] and by $E^{(q)}$ the determinant associated to \[\mathcal{F}_{l+1} = \Syz(x^{aq}, y^{aq}, (x^\delta + y^\delta + t x^{\frac{\delta}{2}} y^{\frac{\delta}{2}})^{l+1})\] (as in the proof of Theorem \ref{PGenericStronglySst}). Let $t_0 \in \overline {k}$ be such that $Y_{t_0}$ is a smooth fibre of $Y \to \Spec k[t]$. Then the following hold:  
\begin{enumerate}[(a)]
\item{If $D^{(q)}(t_0) \neq 0$ and $E^{(q)}(t_0) \neq 0$ for all $q = p^e$ then $\mathcal{S}\vert_{Y_{t_0}}$ is strongly semistable. If $D^{(q')}(t_0) \neq 0$ and $E^{(q')}(t_0) \neq 0$ for some $q' \geq \frac{1}{2} \delta (\delta -1) + 1$ then $\mathcal{S}\vert_{Y_{t_0}}$ is semistable.}
\item{If $D^{(q)}(t_0) = 0$ for some $q = p^e$ with $r =1$ then $F^{e^\ast}\mathcal{S}\vert_{Y_{t_0}}$ is not semistable. In particular, $\mathcal{S}\vert_{Y_{t_0}}$ is not strongly semistable.}
\item{If $\mathcal{F}_l$ has a nonzero global section of total degree $\lceil \frac{3aq}{2} \rceil -r-1 $ or if $\mathcal{F}_{l+1}$ has a nonzero global section of total degree $ \lceil \frac{3aq}{2} \rceil -1 $ for some $q = p^e$ then $F^{e^\ast} \mathcal{S}\vert_{Y_{t_0}}$ is not semistable.}
\end{enumerate}
\end{Le}
\begin{proof}
Part (a) follows via the arguments employed in Theorem \ref{PGenericStronglySst}. The second statement follows from Lemma \ref{KGlobalSectionsSst}. As to part (b), we may assume that $\delta \geq 4$. 
Since $r =1$, sections of critical degree for $\mathcal{F}_{l}\vert_{Y_{t_0}}$ yield sections of critical degree for $F^{e^\ast}\mathcal{S}\vert_{Y_{t_0}}$.

Part (c) is clear since a section of $\mathcal{F}_l$ of total degree $m$ yields a section of total degree $m + r$ in $F^{e^\ast} \mathcal{S} \vert_{Y_{t_0}}$ and a section of $\mathcal{F}_{l+1}$ yields a section of the same total degree. 
\end{proof}

We do not have a general result which guarantees the existence of special elements $t_0 \in \overline{k}$ with the several semistability behaviours described in Lemma \ref{LPseudoAlgorithms}. However, this lemma provides the basis for various computational methods to find such special elements. These methods were implemented in \cocoa\ (see \cite{CocoaSystem}) and Macaulay2 (see \cite{M2}). The computations were made only for $\delta = 4$ and $a=1$, so we restrict to this case.

\medskip \noindent
$\bullet$
Part (b) of Lemma \ref{LPseudoAlgorithms} yields in theory a way to produce examples of bundles with the properties described in the introduction if $p \equiv 1$ mod $\delta$. Fix a prime characteristic $p$ and let $D^{(q)}$ be the determinant of the matrix corresponding to ${\mathcal F}_l$ and $E^{(q)}$ the determinant of the matrix corresponding to $\mathcal{F}_{l+1}$ for $q=p^{e}$ ($a=1$). If $(D^{(q)} \cdot E^{(q)})(t) \neq 0$ for some $q \geq (\delta -1)^2$, then ${\mathcal S}_t$ is semistable. If $D^{(q)} (t)=0$ for some $q=p^e$ (think of a larger $q$), then ${F^e}^\ast \mathcal{S}_t $ is not semistable, hence ${\mathcal S}_t$ is not strongly semistable. So for given $p$ we ``only'' have to look for $t \in \overline{k}$ and $q \geq (\delta -1)^2 $ such that $(D^{(q)} \cdot E^{(q)})(t) \neq 0$ but $D^{(qp)}(t)=0$. As $q$ grows the matrices get larger and so it is natural to expect that the number of zeros grows as well.

We implemented this as follows. Assume that $H^{(q)}(t) := (D^{(q)} \cdot E^{(q)})(t)$ and $D_{0}(t) := D^{(qp)}(t)$ are as above. One divides $D_{i}$ by $\gcd(H^{(q)}, D_{i})$. Call the resulting polynomial $D_{i+1}$, increment $i$ by one and repeat this process until $\gcd(H^{(q)}, D_{i}) = 1$. If the resulting polynomial has positive degree then we find $t_0$ as required. Unfortunately, computing determinants is very expensive. Hence, we were only able to use it for $p = 5, 13$ to get elements $t_0 \in k$ such that the syzygy bundle is semistable but not strongly semistable.

\medskip \noindent
$ \bullet$
Lemmata \ref{HNFMethod} and \ref{LPseudoAlgorithms} (c) provide the following method. Fix a prime number $p$. For all elements $t_0 \in \mathbb{F}_p$ we compute syzygies of $\mathcal{F}_l$, $\mathcal{F}_{l+1}$ (for $q=p$) of minimal degree in $\mathbb{F}_p[x,y]$ (note that it simplifies computational matters substantially to work in two variables)\footnote{We restricted ourselves to prime fields since \cocoa\ does not yet support general finite fields.}. These yield 
syzygies of $\Syz(x^p,y^p,z^p)$ on the curve given by $z^\delta=x^\delta+y^\delta +t_0 x^{\frac{\delta}{2}}y^{\frac{\delta}{2}}$. 

Suppose that such a syzygy $s = (s_1, s_2, s_3)$ of $F^\ast \mathcal{S} = \Syz(x^p,y^p,z^p)$ shows that this bundle is not semistable, which is just a degree condition. Then we check the conditions of Lemma \ref{HNFMethod}. To verify that $s$ has no zeros we have to check that $s_1, s_2, s_3$ generate an $R_+$-primary ideal. This in turn is equivalent to checking that $k[x,y,z]/(x^\delta + y^\delta + t_0 x^{\frac{\delta}{2}}y^{\frac{\delta}{2}} - z^4))/(s_1, s_2, s_3)$ is zerodimensional. Finally, we check that the degree fulfills the divisibility condition. If for some $t_0 \in  \mathbb{F}_p$ all these three conditions are fulfilled, then the syzygy bundle is semistable and its first Frobenius pull-back is not semistable anymore.

With this method we have found examples for most prime numbers $p \leq \primeboundlarge $, but also with many exceptions. For $p \leq 100$ the exceptions are $7, 23, 31, 47, 89$. For $p=7$ we also worked with higher Frobenius pull-backs but this was not successful.

\medskip \noindent
$ \bullet$
For the prime numbers $p =7, 23, 31, 47, 89$ (the only cases $\leq 100$ where the second method failed) we worked with ${\mathbb F}_{p^2}$ instead. This is computationally more expensive as we need a third variable. The second method directly applied to $\mathbb{F}_{p^2}$ is also expensive as we have to run through $p^2$ elements. Hence, we rather looked at the zeros of the determinants and found elements $t_0 \in {\mathbb F}_{p^2}$ with the looked-for behaviour.

\begin{Bsp}
Let $p =3$. Then the determinants associated to \[\mathcal{F}_0 = \Syz(x^3, y^3, (x^4 + y^4 - t x^2 y^2)^0)\] and to \[\mathcal{F}_1 = \Syz(x^3, y^3, (x^4 + y^4 - t x^2 y^2)^{1})\] are $D = 1$ and $E = t$ respectively. So only $E$ has a zero over $\mathbb{F}_3$ and the corresponding curve is the Fermat quartic $x^4 + y^4 - z^4$ over $\mathbb{F}_3$. Then $\mathcal{F}_{1}\vert_{t_0} = \Syz(x^3, y^3, (x^4 + y^4))$ is generated by $(x,y,-1), (y^3, -x^3, 0)$. Hence, we obtain the syzygies $s = (x,y, -z) $ and $t =(y^3, -x^3, 0)$ for $\Syz(x^3, y^3, z^3)$. Both sections do not have a zero since the components generate a primary ideal on the Fermat quartic. The syzygy $s$ is of total degree $4$ and $t$ is a (Koszul-)syzygy of total degree $6$. In particular, the total degree of $t$ fails to satisfy both of the numerical conditions of Lemma \ref{HNFMethod} while the total degree of $s$ satisfies these conditions. Hence, we obtain via Lemma \ref{HNFMethod} that $\Syz(x,y,z)$ is semistable on the Fermat quartic over $\mathbb{F}_3$ while its first Frobenius pull-back is not semistable.
\end{Bsp}

\begin{Bsp}
For $p = 7$ we have $7=4 \cdot 1 +3$ and so we obtain the matrices
\[
\begin{pmatrix} t & 0 & 1 & 0 & 0 \\ 0 & t & 0 & 1 & 0 \\ 1 & 0 & t & 0 & 1 \\ 0 & 1 & 0 & t & 0 \\ 0 & 0 & 1 & 0 & t
\end{pmatrix}
\]
with determinant $D=t^5 +4t^3 + 2t$ (corresponding to $\mathcal{F}_1$) which has the zeros $0, 1,3,4,6$ over $\mathbb{F}_7$
and
\[
\begin{pmatrix} t^2 +2  & 0 & 2t  \\ 0 & t^2 +2 & 0  \\ 2t & 0 & t^2 +2 
\end{pmatrix},
\]
with determinant $E= t^6 + 2t^4 +4 t^2 + 1$ which factors as $(t^2 + 2)(t^2 +5t + 2)(t^2 + 2t + 2)$. In particular, $E$ has no zeros over the prime field.
However, all syzygies for $\Syz(x^7, y^7, z^7)$ that we obtained via $\mathcal{F}_1$ fail to be of total degree $< - \frac{p \deg \mathcal{S}}{2 \deg \mathcal{O}_Y(1)} = \frac{21}{2}$. In order to settle the case $p = 7$ we then looked at the fibre $t_0$ in $\mathbb{F}_{49}$ satisfying $t_0^2 + 2 = 0$ and directly computed generators of $\Syz(x^7, y^7, z^7)$ on the curve $z^4 = x^4 + y^4 + t_0 x^2 y^2$. In particular, this yielded the syzygy $(-2y^3t_0 - x^2y, -2x^3t_0 - xy^2, xyz)$ which is primary and satisfies all the numerical conditions of Lemma \ref{HNFMethod}.
\end{Bsp}

\begin{Bem}
\label{ETS}
In order to prove that for all (or at least for infinitely many) prime numbers there exist elements $t_0 \in \overline{{\mathbb F}}_p$ such that $\Syz(x,y,z)$ is semistable but not strongly semistable on the corresponding curve one needs probably a better understanding of the structure of the matrices occuring in the proof of Theorem \ref{PGenericStronglySst} and their determinants. By the structure theorem on finitely generated modules over principal ideal domains, these matrices are similar to diagonal matrices \[\begin{pmatrix} Q_1 &0 & \ldots & 0 \\ 0 &Q_2 & \ldots & 0 \\ \vdots & \ldots  & \ddots &  \vdots  \\ 0 &  \ldots & 0 & Q_s  \end{pmatrix} \, ,\]
with polynomial entries $Q_i$ where $Q_i$ divides $Q_{i+1}$. For $t _0 \in \overline{{\mathbb F}}_p$, the number of polynomials occuring on the diagonal which vanish at $t_0$ is directly related to the minimal degree $k$ where ${\mathcal F}_l(k)$ (or ${\mathcal F}_{l+1}(k)$) has a nonzero section and hence via Lemma \ref{LPseudoAlgorithms} (c) to the destabilizing behaviour of $F^{e^\ast} \Syz$. We do neither know an explicit description of the polynomials occuring on the diagonal nor what it means when $t_0$ is a zero of higher order of some of these polynomials. 
\end{Bem}

Combining the generic result of Section \ref{varycurvegeneric} and the computations we obtain the following class of examples.

\begin{Prop}
\label{Beispiel1}
Let $k$ be a field of positive characteristic $p$ and consider the syzygy bundle $\mathcal{S} = \Syz(x,y,z)$ on the curve \[Y = \Proj k[t][x,y,z]/(x^4 + y^4 + t x^2 y^2 - z^4) \rightarrow \Spec k[t].\]
Then $\mathcal{S}$ is strongly semistable on the generic fibre $Y_{k(t)}$. For all primes $p \leq \primeboundlarge $, $p \neq$ \primeexceptions , there is a closed point $t_0  \in \Spec k[t] $ with smooth fibre $Y_{t_0}$ such that the special bundle $\mathcal{S}_{t_0}$ on $Y_{t_0}$ is semistable but not strongly semistable. 
\end{Prop}
\begin{proof}
The first claim is immediate from Theorem \ref{PGenericStronglySst}. To check semistability we applied the method based on Lemmata \ref{HNFMethod} and \ref{LPseudoAlgorithms} (c) to $\mathcal{S}$.
\end{proof}

\begin{Bem}
\label{Bblubb}
\begin{enumerate}[(a)]
\item{The bundles in Proposition \ref{Beispiel1} do not have degree $0$, but this can be remedied easily. Consider the finite covering
\[ \varphi: X = \Proj k[t][u,v,w]/(u^{8} + v^{8} + t u^{4} v^{4} - w^{8}) \longrightarrow Y \]
induced by $x \mapsto u^2, y \mapsto v^2$ and $w \mapsto z^2$ and note that $\varphi^\ast \mathcal{O}_Y(1) = \mathcal{O}_X(2)$. Hence, we obtain that $\deg \varphi = 4$. Consequently,
\[ \varphi^\ast \mathcal{S} \otimes \mathcal{O}_X(3) = \Syz(u^2, v^2, w^2)(3) \]
is of degree zero.

Moreover, $\mathcal{S}$ is still generically strongly semistable. Since $\varphi$ is separable, $\varphi^\ast \mathcal{S}$ is still semistable but not strongly semistable for some special fibre for any $p$ as in Proposition \ref{Beispiel1}. Thus passing to a finite separable covering we find examples where $\mathcal{S}$ is of degree zero.}
\item{Passing to a quartic field extension of $k(t)$ and then applying a suitable automorphism of $\mathbb{P}^1$ we may rewrite our curve equation as $z^4 - xy(x+y)((\frac{2a}{a+b} -1)x + \frac{a+b}{a-b}y)$, where $\pm a, \pm b$ denote the zeros of $x^4 + tx^2 + 1$. Write $A = (\frac{2a}{a+b} -1)$ and $B = \frac{a+b}{a-b}$. Again passing to a field extension, transforming $z \mapsto A^{-\frac{1}{4}} z$ and replacing $t$ by $s = \frac{B}{A}$ as a variable we have an isomorphism with $z^4 - xy(x+y)(x + sy)$}. 
\end{enumerate} 
\end{Bem}

\begin{Bem}
It is worth mentioning that similar examples already occured in characteristic $2$ and $3$ in Hilbert-Kunz theory. More precisely, Paul Monsky showed in \cite[Theorem III]{monskyzdp4} that the Hilbert-Kunz multiplicity of $(x,y,z)$ in the fibre ring
\[ \mathbb{F}_3[t][x,y,z]/(z^4 - xy(x+y)(x+ ty) )\] is $3 + \frac{1}{9^d}$ or $3$ depending on whether $t$ is algebraic or not, where $d = (\mathbb{F}_3(t):\mathbb{F}_3)$. 

Via the interpretation of Hilbert-Kunz multiplicities in terms of vector bundles over curves one obtains that $\mathcal{S} = \Syz(x,y,z)$ is strongly semistable if and only if its Hilbert-Kunz multiplicity is $3$ (see \cite[Corollary 4.6]{brennerhilbertkunz}). Since semistability is an open property one thus obtains special fibres where $\mathcal{S}$ is semistable but not strongly semistable. Also note that the first example of Monsky in characteristic $3$ is generically isomorphic to our example after passing to a finite field extension -- this follows from the discussion in Remark \ref{Bblubb} (b).

In characteristic $2$, the results of \cite{monskypoints4quartics} say that the Hilbert-Kunz multiplicity of
\[{\mathbb F}_2 [t][x,y,z]/(z^4 + xyz^2 + x^3z + y^3z + t x^2y^2)\]
depends on whether $t$ is algebraic or not. This in turn implies that the syzygy bundle $\Syz(x,y,z)$ on
\[ \Proj \mathbb{F}_2[t][x,y,z]/(z^4 + xyz^2 + x^3z + y^3z + t x^2y^2) \longrightarrow \Spec {\mathbb F}_2[t] \]
is generically strongly semistable and semistable but not strongly semistable for certain (in fact all smooth) special fibres.
\end{Bem}

\begin{Bem}
Since Proposition \ref{Beispiel1} is not applicable in characteristic $2$ we provide a separate example for this case. Consider $\mathcal{S} = \Syz(x^2, y^2,z^2)(3)$ on the smooth relative curve $Y = \Proj k[t][x,y,z]/(x^5 + y^5 - z^5 + tx^2y^3)$, where $k$ is a field of characteristic $2$. We denote the generic fibre of $\mathcal{S}$ by $\mathcal{S}_t$ and the special fibre at $0$ by $\mathcal{S}_0$.

First, we claim that $\mathcal{S}_0$ is semistable but not strongly semistable. 
The first Frobenius pullback of $\Syz(x^2,y^2, z^2)$ on the fibre $t = 0$ admits the syzygy $s = (x,y,z)$ in total degree $5$. Hence, $\mathcal{S}_0$ is semistable but not strongly semistable by Lemma \ref{HNFMethod}. 

Next, we claim that $\mathcal{S}$ is strongly semistable on the generic fibre. In fact, $F^{3^\ast}(\Syz(x^2,y^2,z^2)_t)$ has the (generically) linearly independent global sections
\begin{align*}
s_1 &= (x^4 y^4, x^8 t^4 + y^8, y^4 z^4) \text{ and }\\
s_2 &= (x^2 y^5 z t^5 + y^7z t^4 + x^2 y^5 z + x^2 z^6, x^6yzt^9 + x^4y^3zt^8 + x^2y^5zt^7 + x^5y^2zt^6 +\\& y^7zt^6 + x^6yzt^4 + xy^6zt^4 + x^2y^5zt^2 + x^5y^2zt + y^2z^6t + x^3y^4z,\\
&y^8t^6 + x^3y^5t^5 + xy^7t^4 + x^2y^6t^2 + y^8t + x^8 + x^3y^5)
\end{align*}
of total degree $24$. Observe furthermore that $s_1$ has no zeros. Hence, $F^{3^\ast} (\mathcal{S}_t)$ is trivial and in particular strongly semistable. It follows that $\mathcal{S}_t$ is strongly semistable as well.
\end{Bem}

From Proposition \ref{Beispiel1} we also immediately obtain

\begin{Ko}
Consider the twisted syzygy bundle $\mathcal{S} = \Syz(u^2, v^2, w^2)(3)$ on the relative curve \[\Proj V[u,v,w]/(u^8 + v^{8} + t u^{4} v^{4} - w^{8}) \longrightarrow \Spec V\] over the discrete valuation ring $V = K[t]_{(t - t_0)}$, where $K$ is a field of positive characteristic $p \leq \primeboundlarge$, with the exceptions mentioned in Proposition \ref{Beispiel1}, such that $\mathbb{F}_{p^2} \subseteq K$ and $t_0$ is a suitable element of $K$.
Then $\mathcal{S}$ is of degree zero and generically strongly semistable and it is semistable but not strongly semistable in the special fibre.
\end{Ko}
\begin{proof}
Fix a fibre as in Proposition \ref{Beispiel1} and pull back to the curve given in Remark \ref{Bblubb} (a). Replacing $k$ by a suitable quadratic extension if necessary we may assume that this fibre corresponds to a rational point so that its ideal is of the form $(t- t_0)$. Localising thus yields the result.
\end{proof}

\bibliography{bibliothek.bib}
\bibliographystyle{amsplain}
\end{document}